\documentclass[10pt,a4paper,reqno]{amsart} 

\usepackage{amsmath}
\usepackage{bbm}
\usepackage{mathtools}
\usepackage{amssymb}
\usepackage{graphicx}
\usepackage{color}
\usepackage{latexsym}
\usepackage[utf8]{inputenc}
\usepackage[T1]{fontenc}
\usepackage{comment}
\usepackage{stmaryrd}

\usepackage{float}

\newcommand{\ns}{{\mathbb N}} 
\newcommand{\qs}{{\mathbb Q}}  
\newcommand{\cs}{{\mathbb C}} 

\newcommand{\al}{\alpha}

\newcommand{\si}{\sigma}
\newcommand{\la}{\lambda}
\newcommand{\eps}{\varepsilon}

\newcommand{\bu}{\bar u}

\newcommand{\tam}{Tam}

\newcommand{\Sn}{{\mathfrak S}}

\newcommand{\GK}{\mathbb{K}}
\newcommand{\GL}{\mathbb{L}}

\DeclareMathOperator{\DR}{DR}

\DeclareMathOperator{\id}{id}
\DeclareMathOperator{\Park}{Park}
\DeclareMathOperator{\Tam}{Tam}

\newcommand{\tG}{\tilde G}
\newcommand{\tz}{\tilde z}

\newcommand{\cL}{\mathcal L}

\newcommand{\cT}{\mathcal T}

\newcommand{\ps}{permutations}

\newcommand{\figeps}[3]
{\begin{figure}[ht!]
\begin{center} 
\includegraphics[width=#1cm]{#2.eps}\caption{#3}\label{fig:#2} 
\end{center}
\end{figure}}

\newtheorem{Theorem}{Theorem}
\newtheorem{Proposition}[Theorem]{Proposition}

\newtheorem{Corollary}[Theorem]{Corollary}

\newtheorem{Lemma}[Theorem]{Lemma}

\newtheorem{Definition}[Theorem]{Definition}

\newcommand{\beq}{\begin{equation}}
\newcommand{\eeq}{\end{equation}}

\newcommand{\gf}{generating function}
\newcommand{\gfs}{generating functions}
\newcommand{\fps}{formal power series}

\def\emm#1,{{\em #1}}

\graphicspath{{Figures/}}

\catcode`\@=11
\def\section{\@startsection{section}{1}%
 \z@{.7\linespacing\@plus\linespacing}{.5\linespacing}%
 {\normalfont\bfseries\scshape\centering}}

\def\subsection{\@startsection{subsection}{2}%
  \z@{.5\linespacing\@plus\linespacing}{.5\linespacing}%
  {\normalfont\bfseries\scshape}}

\def\subsubsection{\@startsection{subsubsection}{3}%
 \z@{.5\linespacing\@plus\linespacing}{-.5em}
  {\normalfont\bfseries\itshape}}
\catcode`\@=12

\def\cT{\mathcal{T}}
\def\cTn{\cT_n}

\newcommand{\spacebreak}
{\begin{displaymath} \triangleleft \; \lhd \;
\diamond \; \rhd \; \triangleright
  \end{displaymath}}

\begin{document}
\title
[An extension of Tamari lattices]
{An extension of Tamari lattices}

\author[L.-F. Préville-Ratelle]{Louis-François Préville-Ratelle}
\author[X. Viennot]{Xavier Viennot}

\address{LFPR: Instituto de Matem\'atica y F\'isica, Universidad de Talca, 2 norte 685, Talca, Chile.}
\email{preville-ratelle@inst-mat.utalca.cl}
\address{XV: CNRS, LABRI, Universit\'e de Bordeaux, Bordeaux, France}
\email{viennot@xavierviennot.org}
%

\thanks{LFPR was supported by a Proyecto Fondecyt Postdoctorado 3140298. }


\begin{abstract}

For any finite path $v$ on the square grid consisting of north and east unit steps, starting at (0,0), we construct a poset Tam$(v)$ that consists of all the paths weakly above $v$ with the same number of north and east steps as $v$. For particular choices of $v$, we recover the traditional Tamari lattice and the $m$-Tamari lattice. 

Let $\overleftarrow{v}$ be the path obtained from $v$ by reading the unit steps of $v$ in reverse order, replacing the east steps by north steps and vice versa. We show that the poset Tam$(v)$ is isomorphic to the dual of the poset Tam$(\overleftarrow{v})$. We do so by showing bijectively that the poset Tam$(v)$ is isomorphic to the poset based on rotation of full binary trees with the fixed canopy $v$, from which the duality follows easily. This also shows that Tam$(v)$ is a lattice for any path $v$. We also obtain as a corollary of this bijection that the usual Tamari lattice, based on Dyck paths of height $n$, is a partition of the (smaller) lattices Tam$(v)$, where the $v$ are all the paths on the square grid that consist of $n-1$ unit steps. 

We explain possible connections between the poset Tam$(v)$ and (the combinatorics of) the generalized diagonal coinvariant spaces of the symmetric group.

\end{abstract}

\date{\today}
\maketitle


\section{Introduction}\label{secintroduc}


In this article, we generalize the $m$-Tamari lattice to posets of arbitrary paths, as it is explained in section \ref{sectmainresults}. We prove that these posets are actually lattices, that they satisfy a duality property, and that they partition the ordinary Tamari lattice into intervals. We first introduce some basic definitions in section \ref{secbasicdefs} and some motivations in section \ref{rationalcomb}.

\subsection{Basic definitions}\label{secbasicdefs}

A binary tree is defined recursively as follows. A binary tree $T$ is either the empty set, or else a triple $(L,r,R)$ where $L$ and $R$ are binary trees and $r$ is the root vertex of $T$. The binary trees $L$ and $R$ are called respectively the left and the right subtrees of $r$. The root of $L$ (respectively $R$) is called the left (respectively right) child of $T$. The degree of a vertex is its number of children, which is either 0,1 or 2 for all the vertices of a binary tree. An external vertex is a vertex with degree 0 and an internal vertex is a vertex of degree 1 or 2. An external edge is an edge adjacent to an external vertex. For $s$ a vertex in $T$ with left and right subtrees $L_s$ and $R_s$, the subtree at $s$ in $T$ is the (binary) tree $(L_s$,s,$R_s)$.  The vertices in the subtree at $s$ are the descendants of $s$. A complete binary tree is a binary tree such that all the vertices have degrees 0 or 2. For more on trees and other combinatorial structures see \cite{stanley-vol2}.

\begin{figure}
  \centering
  \includegraphics[width=8cm]{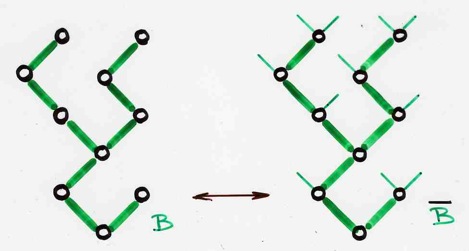}
  \caption{A binary tree B and its associated complete binary tree $\overline{B}$.}
  \label{fig:treeandcompletion}
\end{figure}

It is easy to establish a bijection between binary trees with $n$ vertices and complete binary trees with $n$ internal vertices. It consists of completing a binary tree by adding $n+1$ edges so that all the initial vertices become internal vertices and it becomes a complete binary tree. We denote by $\overline{B}$ the completion of the binary tree $B$. And conversely, if $\overline{B}$ is a complete binary tree, $B$ is obtained from $\overline{B}$ by deleting its external vertices. An example of this notation is given in Figure \ref{fig:treeandcompletion}. 
These two families of trees are enumerated by the well studied Catalan numbers $C_n = \frac{1}{2n+1} \binom{2n+1}{n}$. 

We now define the Tamari lattice. The complete binary trees with $n$ interior vertices can be equipped with a rotation. As in Figure \ref{fig:rotationbinary}, consider a complete binary tree $\overline{T}$ with an internal vertex $s$ such that the left child of $s$, denoted by $t$, is also an internal vertex. Let $A$ be the left subtree of $t$, $B$ the right subtree of $t$ and $C$ the right subtree of $s$. Let $\overline{T'}$ be the complete binary tree constructed from $\overline{T}$ such that $t$ becomes the right child of $s$, $A$ the left subtree of $s$, $B$ the left subtree of $t$ and $C$ the right subtree of $t$. This operation from $\overline{T}$ to $\overline{T'}$ is called a right rotation, and the operation from $\overline{T'}$ to $\overline{T}$ is called a left rotation. In fact, the covering relations of the well known Tamari lattice (see \cite{tam1962,friedman-tamari} ) are the relations $\overline{T}<\overline{T'}$ . 

\begin{figure}
  \centering
  \includegraphics[width=9cm]{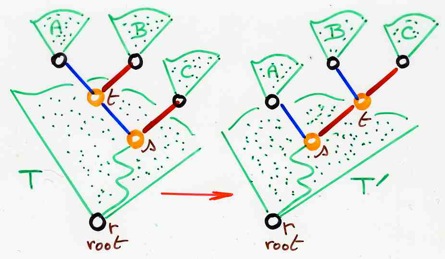}
  \caption{Right rotation on a complete binary tree: the covering relation in the Tamari lattice.}
  \label{fig:rotationbinary}
\end{figure}

\begin{figure}[hb] 
  \centering
  \includegraphics[width=9cm]{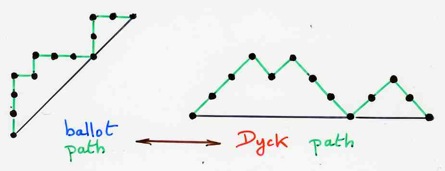}
  \caption{Ballot path and Dyck path.}
  \label{fig:ballotanddyck}
\end{figure}

In this article, we consider a path to be a (finite) walk on the square grid, starting at (0,0), consisting of north and east unit steps denoted by $N$ and $E$ respectively. The set of ballot paths of height $n$ is the set of paths that consist of $n$ north steps, $n$ east steps and weakly above the diagonal, that is, weakly above  the path $(NE)^n$. They are also counted by the Catalan numbers. By applying a clockwise rotation of 45 degrees on ballot paths so that the diagonal becomes horizontal, these ballot paths become the well known Dyck paths (see Figure \ref{fig:ballotanddyck}). The ballot paths can be generalized with a parameter $m$ that is a positive integer. The $m$-ballot paths are the paths that consist of $n$ north steps, $mn$ east steps and weakly above the line $y=\frac{x}{m}$, that is, weakly above the path $(NE^m)^n$. 

Using the bijection between complete binary trees with $n$ internal vertices and ballot paths of height $n$, which is called the post order traversal on edges, the covering relations for the Tamari lattice can be translated into the following procedure on ballot paths. Let $D$ be a ballot path of height $n$. Let $E$ be an east step that precedes a north step in $D$. Draw a diagonal of slope 1 starting at the right extremity of $E$ until it touches $D$ again. Construct $D'$ from $D$ by switching E and the portion of the path above this diagonal. Then the covering relation in the Tamari lattice based on ballot paths becomes $D<D'$ (see Figure \ref{fig:Tamaricoveringm1et2} for such a covering relation).

\begin{figure}
  \centering
  \includegraphics[width=9cm]{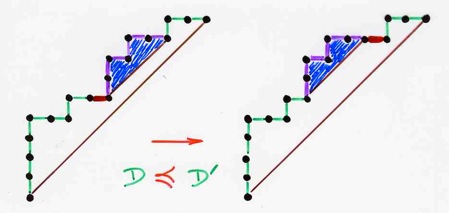}
  \\ 
  \includegraphics[width=9cm]{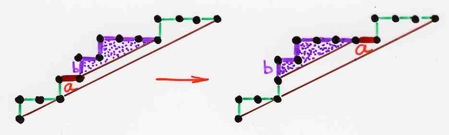}
  \caption{The Tamari covering relation for ballot (Dyck) path (top figure). The covering relation in the $m$-Tamari lattice ($m$=2) (bottom figure).}
  \label{fig:Tamaricoveringm1et2}
\end{figure}

Motivated by the higher diagonal coinvariant spaces of the symmetric group, the covering relation on ballot paths is generalized in \cite{bergeron-preville} to $m$-ballot paths by mimicking the above procedure as follows. Let $D$ be an $m$-ballot path. Let $E$ be an east step that precedes a north step in $D$. Draw a diagonal of slope $\frac{1}{m}$ starting at the right endpoint of $E$ until it touches $D$ again. Construct $D'$ from $D$ by switching E and the portion of the path above this diagonal. Then the covering relation in the $m$-Tamari lattice is given by $D<D'$ (see Figure \ref{fig:Tamaricoveringm1et2} for an example). For more on these lattices and for enumerations of their intervals, we refer the reader to section \ref{seccoinvariantperspec}.


\subsection{Rational Catalan combinatorics $(a,b)$}\label{rationalcomb} 

Let $a$ and $b$ be two relatively prime integers. We consider paths starting at $(0,0)$ on the square grid with north and east steps and strictly above the line $y = \frac{a}{b} x$, excluding the start and end points (see \cite{Biz54}). An example is given in Figure \ref{fig:ab_ballot}. They are called $(a,b)$-ballot paths (or $(a,b)$-Dyck paths), and their study is the subject of very recent work under the term ``rational Catalan combinatorics'' (see \cite{Hikita,GoNe13,ArmstRhoaWilli,GoMaVa14,ArLoWa14} for more on this subject).
The classical ballot paths and their extensions with any integer $m$ are particular cases of such $(a,b)$-ballot paths. The simple Catalan ballot paths are obtained by putting $(a,b)=(n,n+1)$, and their $m$-extensions are obtained by putting $(a,b)=(n,mn+1)$, as shown in Figure \ref{fig:dyck_mdyck}.

\begin{figure}[h]
  \centering
  \includegraphics[width=4cm]{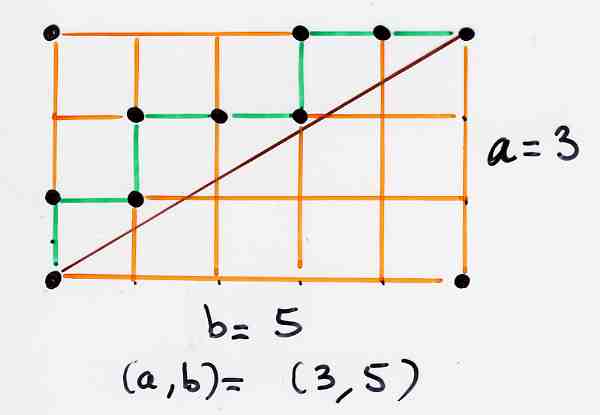}
    \caption{A (3,5)-ballot path.}
  \label{fig:ab_ballot}
\end{figure}

\begin{figure}[h]
  \centering
  \includegraphics[width=4cm]{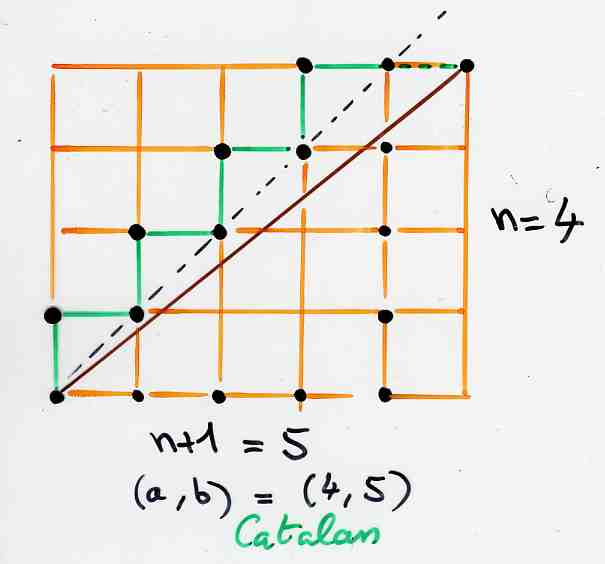}\hspace{5.5mm} \includegraphics[width=4.5cm]{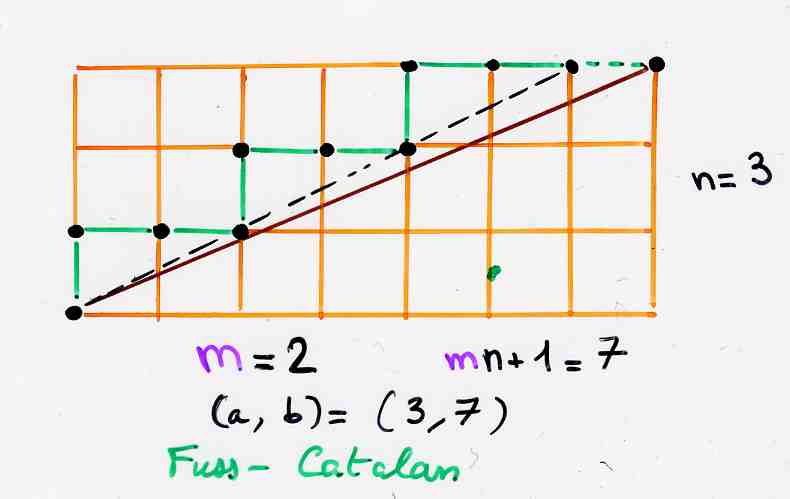}
    \caption{A (4,5)-ballot path (left). A (3,7)-ballot path (right).}
  \label{fig:dyck_mdyck}
\end{figure}

An open question is to give an extension of the Tamari lattice, and more generally of the $m$-Tamari lattice to any pair $(a,b)$ of relatively prime integers. We propose an answer to this question, by giving a far more general extension of these Tamari lattices and in particular give a construction of a rational $(a,b)$-Tamari. 


\subsection{Extension: The Tamari lattice \tam$(v)$, where $v$ is an arbitrary path}\label{sectmainresults}
Let $v$ be an arbitrary path, starting at (0,0). Consider all the unitary paths weakly above $v$ that start at (0,0) and finish at the end of $v$. We define the poset \tam$(v)$ on this set of paths with a covering relation. Let $u$ be such a path above $v$. Let $p$ be an integer point on $u$. We define the horizontal distance horiz$_v(p)$ to be the maximum number of east steps that can be added to the right of $p$ without crossing $v$ (an example of these horizontal distances is given in Figure \ref{fig:horizdistance}). 
\begin{figure}  [h]               
  \centering
  \includegraphics[width=5.8cm]{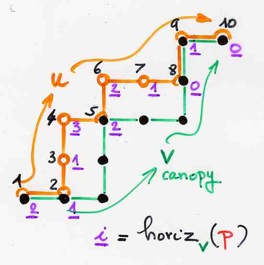}
  \caption{A pair (u,v) of paths with the horizontal distance $horiz_v(p)$.}
  \label{fig:horizdistance}
\end{figure}
Suppose that $p$ is preceded by an east step $E$ and followed by a north step in $u$. Let $p'$ be the first unitary point in $u$ that is after $p$ and such that horiz$_v(p')$ $=$ horiz$_v(p)$. As in Figure \ref{fig:coveringarbitraire}, let $D_{[p,p']}$ be the subpath of $u$ that starts at $p$ and finishes at $p'$. Let $u'$ be obtained from $u$ by switching $E$ and $D_{[p,p']}$. We define the covering relation to be $u <_v u'$ (see Figure \ref{fig:coveringarbitraire} for an example). Then the poset Tam$(v)$ is the transitive closure $<_v$ of this relation. It is easy to see that Tam$((NE^m)^n)$ is the $m$-Tamari lattice.

\begin{figure}    [h]            
  \centering
  \includegraphics[width=5.8cm]{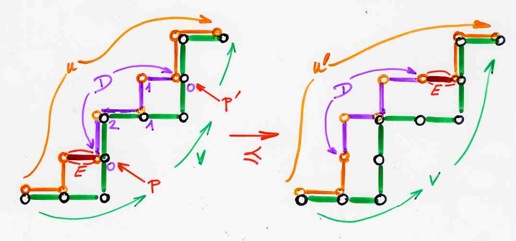}
  \caption{The covering relation defining the poset Tam$(v)$.}
  \label{fig:coveringarbitraire}
\end{figure}

For $v$ an arbitrary path, let $\overleftarrow{v}$ be the path obtained by reading $v$ backward and replacing the east steps by north steps and vice versa. We can now state our main results:

\begin{Theorem}\label{thmlattice}
For any path $v$, Tam$(v)$ is a lattice.
\end{Theorem}

\begin{Theorem}\label{thmisomdual}
The lattice Tam$(v)$ is isomorphic to the dual of Tam$(\overleftarrow{v})$.
\end{Theorem}
An example of this duality is given in Figure \ref{fig:duality}.

Recall from the first section that the usual Tamari lattice on complete binary trees with $n$ interior vertices is isomorphic to the lattice Tam$((NE)^n)$.
\begin{Theorem}\label{thmpartitiontamlat}
The usual Tamari lattice Tam$((NE)^n)$ can be partitioned into disjoint intervals ${I}(v)$ indexed by the unitary paths $v$ consisting of a total of $n-1$ east and north steps, i.e. 
$$ Tam((NE)^n) = \bigcup_{|v|=n-1} I(v),$$
where each ${I}(v) \cong {Tam}(v)$. 
\end{Theorem}
An example of Theorem \ref{thmpartitiontamlat} is given in Figure \ref{fig:partitionlattice}.

\begin{figure}  [H]  
  \centering
  \includegraphics[width=8cm]{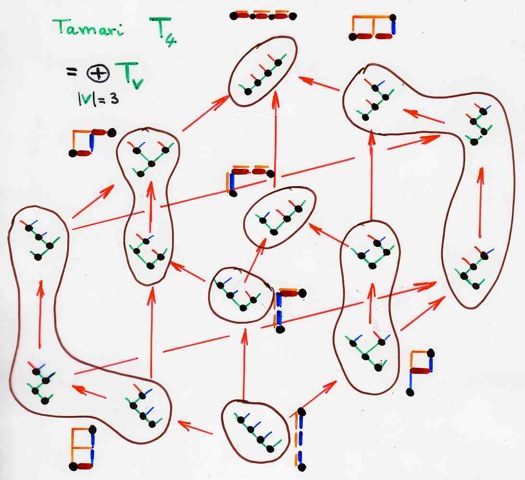}
  \caption{The decomposition of the Tamari lattice on complete binary trees with 4 interior vertices into the union of 8 disjoint intervals Tam$(v)$ (Theorem \ref{thmpartitiontamlat}).}
  \label{fig:partitionlattice}
\end{figure}

\newpage 

\begin{figure}  [H]     
  \centering
  \includegraphics[width=7.6cm]{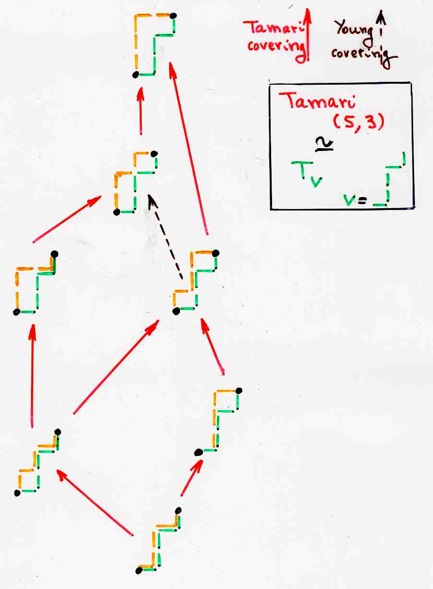} \hspace{3mm} 
  \includegraphics[width=4.5cm]{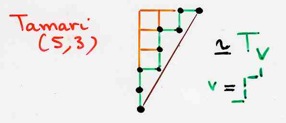}
  \includegraphics[width=6.1cm]{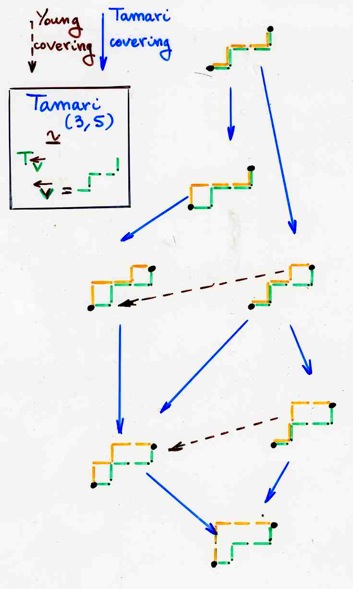}
 \hspace{3mm} \includegraphics[width=5.2cm]{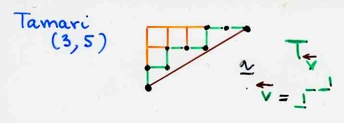}

   \caption{The lattice Tamari(5,3) is the lattice based on the minimum path above the segment passing through the origin and the point (3,5)  (top). The dual of the lattice Tamari(3,5) is the dual of the lattice based on the minimum path above the segment passing through the origin and the point (5,3)  (bottom). Note that Tamari(5,3) is isomorphic to the dual of Tamari(3,5).}
  \label{fig:duality}
\end{figure}

\subsection{Outline}

In Section \ref{canopybinarytree}, we define the concept of the canopy of a binary tree. We state a transformation (a bijection) from the set of  binary trees with $n$ vertices to the set of pairs of weakly non-crossing paths that contain $n-1$ edges each such that they share the same endpoints. The canopy of a binary tree is a word, which we identify with a path made from an alphabet based on two letters, and that is mainly used to partition the Tamari lattice into intervals (sublattices) that are of importance to us. These sublattices are isomorphic to the lattices Tam$(v)$ mentioned in Section \ref{sectmainresults}. The fact that the transformation mentioned previously is a bijection is proved in Section \ref{secreversebijection}. Some simple and useful properties of the canopy of binary trees related with the Tamari lattice are given in Section \ref{sexcanopyordinaryTamari}. The demonstrations of the main results presented in Section \ref{sectmainresults} appear in Section \ref{secproofsmainresults}. Finally, the connections of our work with the diagonal coinvariant spaces and some perspectives appear in Section \ref{seccoinvariantperspec}.

\section{Canopy of a binary tree}\label{canopybinarytree}

For any binary tree $B$, we construct a word $w(B)$ on the alphabet $\{ a,\bar{a},b,\bar{b} \}$. Walking clockwise around $B$ and starting at the root, we write the letter $a$ when we walk on a left edge for the first time and $\bar{a}$ when we walk on a left edge for the second time. Similarly we repeat with the letters $b$ and $\bar{b}$ for the right edges (see Figure \ref{fig:walkaroundbinary} for an example). 
\begin{figure}  [hb] 
  \centering
  \includegraphics[width=4.5cm]{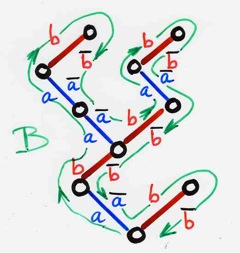}
  \caption{Walk around a binary tree B: the word w(B).}
  \label{fig:walkaroundbinary}
\end{figure}
From $w(B)$, we construct two subwords. The first subword $u(B)$ is obtained by keeping track only of the two letters $\{\bar{a},\bar{b}\}$ in $w(B)$. We identify a path with $u(B)$ by replacing in this sequence the letter $\bar{a}$ by a north step and $\bar{b}$ by an east step. The canopy $v(B)$ of $B$, which is also a sub word of $w(B)$, is obtained similarly by keeping track only of the letters $\{\bar{a},b\}$. We identify a path with $v(B)$ by replacing in this sequence the letter $\bar{a}$ by a north step and $b$ by an east step\footnote{So by abuse of notation, we will refer to the paths $u(B)$ and $v(B)$.}. The concept of the canopy was introduced using a different terminology in \cite{LodayRonco}. For the binary tree in Figure \ref{fig:walkaroundbinary}, we show an example of all these words in Figure \ref{fig:wordsassociatedtree} and draw the paths $u(B)$ and $v(B)$ in Figure \ref{fig:pairpaths}. 
\begin{figure}
  \centering
  \includegraphics[width=9cm]{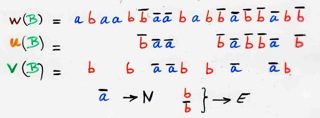}
  \caption{The words $w(B), u(B)$ and $v(B)$ associated to the binary tree in Figure \ref{fig:walkaroundbinary}.}
  \label{fig:wordsassociatedtree}
\end{figure}
\begin{figure}
  \centering
  \includegraphics[width=5.5cm]{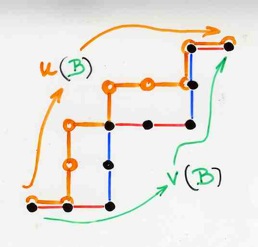}
  \caption{The pair (u,v) of paths associated to a binary tree in Figure \ref{fig:walkaroundbinary}.}
  \label{fig:pairpaths}
\end{figure}
It is no accident that we use the same letter $v$ for the canopy as the letter that defines the poset in the previous section. Before explaining this, we can mention an easy property:

\begin{Lemma}
For any given binary tree $B$, the path $u(B)$ is weakly above the canopy (also a path) $v(B)$.
\end{Lemma}

\begin{proof}
This is straightforward since the occurrences of the letter $b$ precede the occurrences of the letter $\bar{b}$ in $w(B)$, and these two letters determine the east steps of $v(B)$ and 
$u(B)$ respectively. 
\end{proof}


Let $\overline{B}$ be a complete binary tree with $n$ vertices. It is not difficult to prove that the canopy can also be defined using the following two equivalent definitions.

The second definition of the canopy, which is defined in \cite{LodayRonco}, can be described as follows. Walking around $\overline{B}$ clockwise starting at the root, record the sequence of left and right external edges, except the first and last external edges. From this sequence, construct a path by changing the right external edges into north steps and the left external edges into east steps. The path obtained is also the canopy (see Figure \ref{fig:canopythirddefinition} for an example). Because of this definition, we define the interior canopy of the complete binary tree $\overline{B}$ to be the canopy of $B$.
\begin{figure}  [H]
  \centering
  \includegraphics[width=4.5cm]{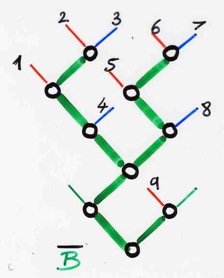}
  \caption{Second definition of the canopy.}
  \label{fig:canopythirddefinition}
\end{figure}

For the third definition of the canopy, we first label the $n$ vertices of the binary tree $B=(L,r,R)$ by the sequence of integers following $B$ according to the symmetric order. This order is defined recursively by first visiting the vertices on the left subtree $L$, then the root $r$, then finishing with the vertices on the right subtree $R$ (see Figure \ref{fig:canopyseconddefinition}). Construct a path from the sequence of vertices with labels $\{ 1,2,...,n-1 \}$ in $B$ such that the $i^{th}$ step is a north step if the vertex with label $i$ in $B$ has a right child, and an east step otherwise. This path is also the canopy. An example of this procedure is given in Figure \ref{fig:canopyseconddefinition}.  
\begin{figure}[H]
  \centering
  \includegraphics[width=3.5cm]{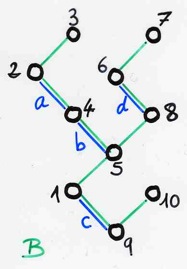}\hspace{4mm} \includegraphics[width=4cm]{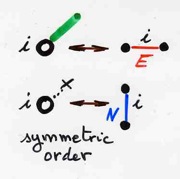}\hspace{4mm} \includegraphics[width=4.5cm]{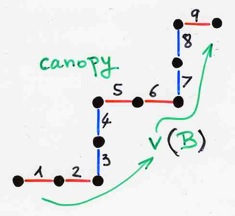}
  \caption{Third definition of the canopy (left and middle). The canopy (right).}
  \label{fig:canopyseconddefinition}
\end{figure}

\section{Reverse bijection: from pairs of non-crossing paths to complete binary trees}\label{secreversebijection}

We will prove the following proposition at the end of this section.
\begin{Proposition}\label{Prop_uv_bijection}
The map defined in Section \ref{canopybinarytree} associating the pair $(u,v)$ to a binary tree is a bijection from the set of binary trees with $n$ vertices to the set of unordered pairs of non-crossing paths, consisting each of a total of $n-1$ north and east steps, with the same endpoints.
\end{Proposition}

Before defining the inverse bijection, we need to state an equivalent definition of the pair $(u,v)$ associated to $B$. For a binary tree ${B}$, let $v(B)$ be the canopy. The left edges in $B$ can be totally ordered using the symmetric order of their fathers in $B$. Note that under this order on left edges, the $i^{th}$ left edge corresponds to the $i^{th}$ north step in both $u(B)$ and $v(B)$. The right height of a vertex in $B$ is the number of right edges of the walk on $B$ from the root to the vertex. The right height of a left edge in $B$ is the right height of any of its vertices. We obtain the following simple fact, which when combined with the definition of the canopy $v(B)$, gives an equivalent definition of the pair $(u,v)$ associated to $B$.

\begin{Lemma}\label{lemunitsquaresbetween}
For $B$ a binary tree, the number of unit squares between the $i^{th}$ north step of $u(B)$ and the $i^{th}$ north step of $v(B)$ is equal to the right height of the $i^{th}$ left edge (in the symmetric order) of $B$ (see Figure \ref{fig:distancescheminscanope} for an example).  
\end{Lemma}

\begin{figure}[H]
  \centering
  \includegraphics[width=5.5cm]{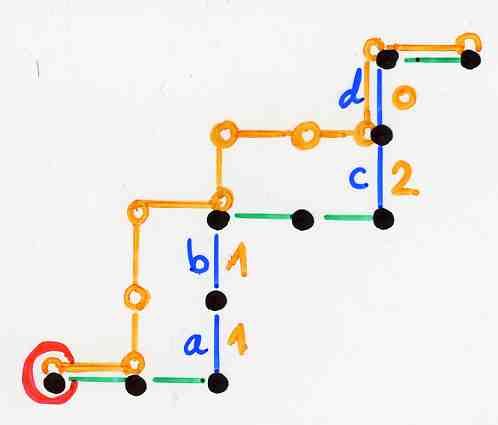}
  \caption{(with Figure 12 (left)) The lemma 6.}
  \label{fig:distancescheminscanope}
\end{figure}

\begin{proof}
In the definition of $w$, we walk clockwise around the binary tree $B$. We record an east step in $v(B)$ when we walk around a right edge of $B$ for the first time and an east step in $u(B)$ when we walk around a right edge of $B$ for the second time. To count the number of unit squares between the $i^{th}$ north steps of $u(B)$ and $v(B)$, we need to count how many right edges have been walked around once time when we walk on the $i^{th}$ left edge of $B$ for the second time. It is easy to show that this corresponds to the number of right edges on $B$ from the root to the $i^{th}$ left edge of $B$, which is the definition of the right height of this edge.
\end{proof}

We now give the reverse transformation $(u,v) \rightarrow B$. Start with a pair of non-crossing paths $D'$ and $D$ with the same endpoints such that $D'$ is weakly above $D$. Draw $D$ and for each north step of $D$, assign the value $a_i$ to be the distance between the $i^{\rm {th}}$ north step in $D$ and the $i^{\rm {th}}$ north step in $D'$. Now, repeat the following process for each north step in $D$, from bottom to top (see Figure \ref{fig:pushglidingalgo} for a better understanding). 

The "push-gliding" algorithm on pairs of non-crossing paths (with the same endpoints) is defined as follows. First let the initial vertex of $D$ be the root. Slide the first north step of $D$ horizontally with all the vertices and all the edges above him so that it is at distance $a_1$ from the root. Let $s_1$ be the vertex at the bottom of this first north step. Let $e_1$ be the adjacent edge of $s_1$ to its left (if any). This first north step pushes $s_1$ down and all the vertices and the edges that are connected to $s_1$ when $e_1$ is removed. If $a_1=0$, then the vertex at the top of the first edge becomes the root. Now redo the same procedure with the second north step and so on\footnote{Note that you don't want to create cycles in this procedure, so you might have to make some edges longer to avoid them.}. At the end, you will obtain a binary tree with $n$ vertices. Now apply a reflection to this tree. This is the end of the push-gliding algorithm. Thus, we obtain a proof of Proposition \ref{Prop_uv_bijection}.

\begin{proof} [Proof of Proposition \ref{Prop_uv_bijection}]
We leave to the reader the proof of the fact that the pair $(u,v)$ and the push-gliding algorithm are inverse transformations. Therefore they are both bijections. 
\end{proof}
The bijection between binary trees and pairs of paths $(u,v)$ was introduced in a different form by Delest and Viennot \cite{XV_DeVi84}. We have described here a new version of the bijection which involves a "push-gliding" algorithm, and fits our purpose.
\begin{figure}[H]
  \centering
  \includegraphics[width=13.2cm]{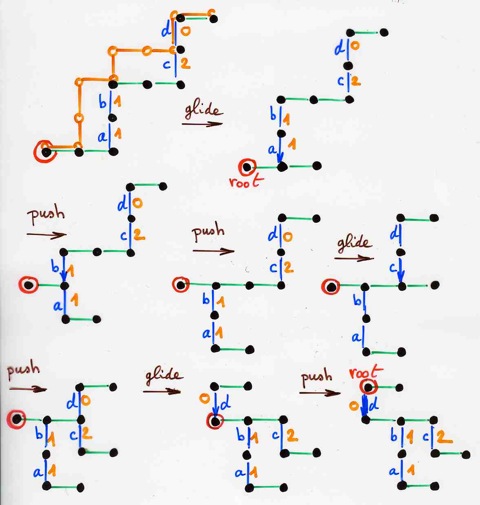}
  \caption{The "push-gliding" algorithm}
  \label{fig:pushglidingalgo}
\end{figure}

\section{Canopy in the ordinary Tamari lattice}\label{sexcanopyordinaryTamari}


For a complete binary tree $\overline{T}$, let $c(\overline{T})$ be the sequence of (external) left edges $\backslash$ and right edges $/$  recorded walking around $\overline{T}$ as in the definition of the interior canopy\footnote{The interior canopy of a complete binary tree is defined in the same paragraph as the second definition of the canopy.}.

\begin{Lemma}\label{lem:canopyrotation}
Let $\overline{T}$ be a complete binary tree, $s$ and $t$ vertices of $\overline{T}$ and $A, B, C$ subtrees of $\overline{T}$ as in Figure \ref{fig:rotationbinary} defining a right rotation from $\overline{T}$ to $\overline{T'}$. Let $c(\overline{T}) = v_1 c(A) / c(C) v_2$ be the sequence of external edges of the interior canopy of $\overline{T}$, where $v_1$ and $v_2$ are some sequences of $\backslash$ and $/$. We have the following relations between the interior canopies of the complete binary trees $A,B,C,\overline{T},\overline{T'}$: \\ 
- if $B$ contains more than one vertex, then $c(\overline{T})=c(\overline{T'})$,\\ 
- if $B$ is a single vertex, then $c(\overline{T'})=v_1 c(A) \backslash c(C) v_2$.
\end{Lemma}

\begin{proof}
This is straightforward using the second definition of the canopy. 
\end{proof}

Thus the interior canopy of a complete binary tree is invariant under a rotation if and only if $B$ contains more than one vertex. We call such a rotation a valid rotation. 

We define the following order relation on the set of words in letters $\{ \backslash, / \}$:
$$ v_1 / v_2 / ... / v_k \prec v_1 \backslash v_2 \backslash ... \backslash v_k ,$$
where the $v_i$ are words of letters $\{ \backslash, / \}$. The poset of words of length $n$ in letters $\{ \backslash, / \}$ is isomorphic to the Boolean lattice of subsets of a set of cardinality $n$.

\begin{Corollary}\label{corol:tamordercanopyorder}
In the Tamari lattice, if $\overline{T} \leq \overline{T'}$, then $c(\overline{T}) \leq c(\overline{T'})$. 
\end{Corollary}

\begin{Proposition}\label{prop_Iv_interval}
The set I$(v)$ of complete binary trees having interior canopy $v$ is an interval of the ordinary Tamari lattice on complete binary trees with $|v|+1$ interior vertices. 
\end{Proposition}

\begin{figure}[H]
  \centering
  \includegraphics[width=5.4cm]{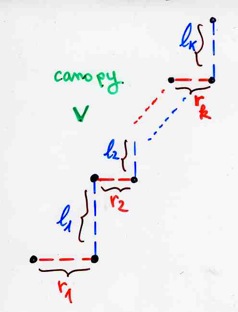}
  \caption{The canopy $v$}
  \label{fig:canopypourarbres}
\end{figure}
\begin{proof}
Let $\overline{T}$ be a complete binary tree with interior canopy $v$, where $v$ is given in Figure \ref{fig:canopypourarbres}. We can always perform a valid right rotation on $\overline{T}$ if and only if there exists vertices $s,t$ with subtrees $A,B,C$ as in Lemma \ref{lem:canopyrotation} (left part of Figure \ref{fig:rotationbinary}). In a right rotation, the length of the left branch rooted in $s$ decreases by one. Starting from $\overline{T}$, by following a sequence of valid rotations, we will always get a tree (with interior canopy $v$) where no further valid right rotations are possible. In such a tree, all the left branches will be reduced to sequences of left edges, where all interior vertices have a right subtree $B$ that consists of a single vertex. This complete binary tree is unique and we denote it by $\overline{T_{\rm max}(v)}$. The binary tree $T_{\rm max}(v)$ is displayed in Figure \ref{fig:arbremaxtamv}.

Similarly, let $\overline{T'}$ be a complete binary tree with canopy $v$. We can always perform a valid left rotation on $\overline{T'}$ if there exists vertices  $s,t$  with subtrees $A,B,C$ as in Lemma \ref{lem:canopyrotation} (right part of Figure \ref{fig:rotationbinary}). In a left rotation, the length of the right branch rooted at  $s$  decreases by one. As in the previous case of left rotation, by following a sequence of valid left rotations, we will always get a unique complete binary tree with canopy $v$ having all the right branches reduced to sequences of right edges, where all interior vertices have a left subtree $B$ that consists of a single vertex. We denote this complete binary tree by $\overline{T_{\rm min}(v)}$.
The binary tree $T_{\rm min}(v)$ is displayed in Figure \ref{fig:arbremintamv}.

We have proved that every complete binary tree $\overline{T}$ with interior canopy $v$ satisfies (for the order relation of the Tamari lattice)  
$$ \overline{T_{\rm min}(v)} \leq \overline{T} \leq \overline{T_{\rm max}(v)}. $$

Conversely, suppose that $\overline{T}$ is a complete binary tree satisfying the above relation. From corollary \ref{corol:tamordercanopyorder}, we have  
$$ v = c(\overline{T_{\rm min}(v)}) \leq c(\overline{T}) \leq c(\overline{T_{\rm max}(v)}) = v. $$
Proposition \ref{prop_Iv_interval} is proved.

\end{proof}
\begin{figure}[hb]
  \centering
  \includegraphics[width=5.4cm]{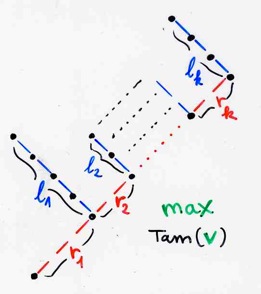}
  \caption{The maximum binary tree $T_{\rm max}(v)$ in Tam$(v)$, where $v$ is given in Figure \ref{fig:canopypourarbres}.}
  \label{fig:arbremaxtamv}
\end{figure}
\begin{figure}[hb]
  \centering
  \includegraphics[width=5.4cm]{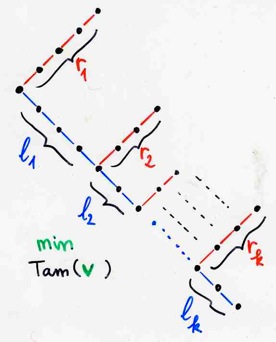}
  \caption{The minimum binary tree $T_{\rm min}(v)$ in Tam$(v)$, where $v$ is given in Figure \ref{fig:canopypourarbres}.} 
  \label{fig:arbremintamv}
\end{figure}



\section{Proofs of the main theorems}\label{secproofsmainresults}

Using the previously defined pair $(u,v)$ that sends a binary tree $T$ to a pair of non crossing paths $(u(T),v(T))$, we will show that the sequence
of values horiz$_v(p)$, for $p$ being the consecutive integer points on $u(T)$ from bottom to top, can be easily read on the vertices of $T$. The post order traversal on vertices of a binary tree $T$ is the order on vertices obtained by walking clockwise around $T$ and recording a vertex the last time we walk next to it (see Figure \ref{fig:arbreverticespostorder} for an example of the post order traversal on vertices). 

\begin{Lemma}\label{lem:seqhorizarbres}
For $T$ a binary tree, the sequence of right heights of the vertices of $T$ in post order traversal corresponds to the sequence of values horiz$_v(p)$, for $p$ being the consecutive integer points on $u(T)$ from bottom to top.
\end{Lemma}
\begin{proof}
Note that walking clockwise around $T$, a vertex is recorded in the post order traversal on vertices just before an edge is visited for the second time (except for the last vertex, which is the root).
This lemma can be proved using the same kind of ideas as for Lemma \ref{lemunitsquaresbetween} and because of the following simple property. Suppose a vertex is recorded in post order traversal in $T$. If you visit a right edge for the first time later, you will then have to visit a left edge for the second time in between (see Figure \ref{fig:arbreverticespostorder} for an example).
\end{proof}
\begin{figure}[H]
  \centering
  \includegraphics[width=4.4cm]{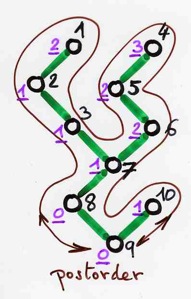}
  \caption{(with Figure \ref{fig:horizdistance}) The Lemma \ref{lem:seqhorizarbres}.}
  \label{fig:arbreverticespostorder}
\end{figure}

\begin{Proposition}\label{prop_isom_IV_Tv}
For any path $v$, the poset I$(v)$ is isomorphic to Tam$(v)$.
\end{Proposition}

\begin{proof}
Let $\overline{T}$ be a complete binary tree with interior canopy $v$. We will show that the covering relations of the two posets are the same (see Figure \ref{fig:coveringrelationssame} for an example). First note that a valid right rotation can be applied exactly in the following situation, and it is easy to show that this situation corresponds bijectively, under the pair $(u(T),v(T))$, to the case where you have an east step preceding a north step in the path $u(T)$, which are precisely the places where the covering relations are defined in Tam$(v)$. Let $s$ and $t$ be vertices of the complete binary tree $\overline{T}$ such that $t$ is the left child of $s$. In $\overline{T}$, let $A$ be the left subtree of $t$, $B$ the right subtree of $t$ that contains more than once vertex, and $C$ the right subtree of $s$. Apply a valid right rotation to $\overline{T}$ to obtain the complete binary tree $\overline{T'}$  (see Figure \ref{fig:coveringrelationssame}). The sequence of right heights of the vertices of $T$ in post order traversal is of the form $(L, S_{A}, S_{B}, h_{t}, S_{C}, h_{s}, R)$, where $S_{A}, S_{B}, h_{t}, S_{C}$,and $h_{s}$ are the sequences of right heights of the vertices in $T$ in post order traversal in $A, B, t, C$ and $s$, respectively, and $L$ and $R$ are the sequences of right heights of all the vertices in $T$ in post order traversal that precede and succeed all the previous vertices respectively. It is clear that $h_{s}=h_{t}$. The sequence of right heights of the vertices in post order traversal of $T'$ is given by $(L, S_{A}, S_{B}, S_{C} + \bar{1}, h_{s}+1, h_{s}, R)$, where $S_{C} + \bar{1}$ corresponds\footnote{Note that $C$ might not contain a vertex of $T$, and therefore this sequence might be empty.} to adding 1 to all the values in $S_{C}$. Since all the values in $S_{C}$, if any, are greater than $h_{s}$, the previous lemma and the definition of the covering relations in Tam$(v)$ imply that the covering relations in I$(v)$ and Tam$(v)$ are the same. 
\end{proof}







We can now prove our main theorems of section \ref{sectmainresults}. 

\begin{proof}[Proof of Theorem \ref{thmlattice}]
An interval of a lattice is always a lattice, therefore I$(v)$ is a lattice by Proposition \ref{prop_Iv_interval}, and so is Tam$(v)$ by Proposition \ref{prop_isom_IV_Tv}.
\end{proof}

\begin{proof}[Proof of Theorem \ref{thmisomdual}]
After applying a reflection to a binary tree with canopy $v$, it is easy to see using the second definition of canopy, that the canopy of this tree (obtained by reflection) is precisely $\overleftarrow{v}$. It is clear from the constructions of Tam$(v)$ and Tam$({\overleftarrow{v}})$ that these lattices are isomorphic up to duality. 
\end{proof}

\begin{proof}[Proof of Theorem \ref{thmpartitiontamlat}]
We partition the complete binary trees with $n$ interior vertices into sets of trees with the same interior canopy. We then apply Proprosition \ref{prop_isom_IV_Tv} to each set of trees.
\end{proof}

\begin{figure}[H]
  \centering
  \includegraphics[width=7.4cm]{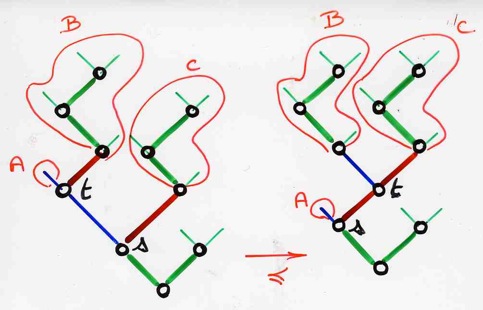}
  \caption{(with Fig 6) The covering relation in Tam$(v)$ and the corresponding rotation in $\overline{T}$ (ordinary Tamari with complete binary trees).}
  \label{fig:coveringrelationssame}
\end{figure}


\section{Connections with the diagonal coinvariant spaces and perspectives}\label{seccoinvariantperspec}

Our work has been influenced by the combinatorics of the ``generalized'' diagonal coinvariant spaces of the symmetric group. We give a brief description of the subject here, and refer the reader to \cite{haglund-book,Bergeron-book-algcomb,bergeron-preville,bousquet-fusy-preville,MbmGcLfpr} for more details. 

Let $X=(x_{i,j})_{^{1 \leq i \leq k}_{1 \leq j \leq n}}$ be a matrix of variables. A permutation $\sigma$ of the symmetric group $\Sn_n$ permutes the variables columnwise by $\sigma(X) = (x_{i,\sigma (j)})_{^{1 \leq i \leq k}_{1 \leq j \leq n}}$, i.e. $\sigma(x_{i,j}) = x_{i,\sigma(j)}$. This action can be directly extended to all the polynomials in $\mathbb{C}[X]$. All the variables in a same row of $X$ are said to be contained in the same set of variables. Since $X$ contains $k$ rows, there are $k$ sets of variables. Let $\mathcal J$ be the ideal generated by constant free invariant polynomials under this action. The diagonal coinvariant spaces of $\Sn_n$ are defined as $\mathcal{DR}_{k,n} := \mathbb{C}[X] \big{/} \mathcal{J} $. They can be generalized using an additional parameter $m$ that is a positive integer. The higher diagonal coinvariant spaces of the symmetric group are defined as $\mathcal{DR}_{k,n}^{m} := {\eps}^{m-1}  \otimes  \mathcal{A}^{m-1} \big{/} \mathcal{J} \mathcal{A}^{m-1}$, where $\epsilon$ is the sign representation and $\mathcal{A}$ be the ideal generated by alternants, i.e. polynomials $f(X)$ such that  $\sigma f(X) = \eps(\sigma) f(X)$, $\forall \sigma \in \Sn_n$. Note that $\mathcal{DR}_{k,n}=\mathcal{DR}_{k,n}^{1}$. The $\mathcal{DR}_{k,n}^{m}$ are representations of $\Sn_n$ because the action given above can be applied to the quotient space $\mathcal{DR}_{k,n}^{m}$. They are graded with respect to the degree of each set of variables. We denote the subspace of alternants of $\mathcal{DR}_{k,n}^{m}$ by ${\mathcal{DR}_{k,n}^{m}}^{\eps}$.

In the case $k=1$, they are classical \cite{Steinberg1964} and the dimensions of ${\mathcal{DR}_{1,n}^{m}}^{\eps}$ and ${\mathcal{DR}_{1,n}^{m}}$ are given by $1$ and $n!$, respectively.

In the case $k=2$, they were first defined and studied by Garsia and Haiman because of their connections with the Macdonald polynomials. It was proven by Haiman \cite{HaimanPreu} that the dimensions of ${\mathcal{DR}_{2,n}^{m}}^{\eps}$ and ${\mathcal{DR}_{2,n}^{m}}$ are given by $\frac{1}{(m+1)n+1} \binom{(m+1)n+1}{mn}$ and $(mn+1)^{n-1}$, respectively. The first number corresponds to the number of $m$-ballot paths of height $n$ and the second one to the number of $m$-parking functions of height $n$. The $m$-parking functions of height $n$ are simply the $m$-ballot paths labelled on the north steps, with the labels in the set $\{ 1,2,...,n \}$ such that consecutive north steps are labelled increasingly. The spaces $\mathcal{DR}_{2,n}^{m}$ have been studied by many researchers for more than 20 years. Despite that, there are still some important unresolved conjectures left in the field. We mention only one here. The $m$-shuffle conjecture \cite{MR2115257} states that the graded Frobenius series of $\mathcal{DR}_{2,n}^{m}$ is equal to a $q,t$-weighted sum on $m$-parking functions, which involves the combinatorial statistics $area$ and $dinv$, and some quasi-symmetric functions associated to these $m$-parking functions. 

For the case $k=3$, Haiman \cite{HaiConj} conjectured in the 1990's that the dimensions of ${\mathcal{DR}_{3,n}}^{\eps}$ and ${\mathcal{DR}_{3,n}}$ are equal to $\frac{2}{n(n+1)} \binom{4n+1}{n-1}$ and $2^n (n+1)^{n-2}$ respectively. 

Independently of all this story, Chapoton \cite{ch06} proved in 2006 that the number of intervals in the Tamari lattice based on complete binary trees with $n$ interval vertices is given by $\frac{2}{n(n+1)} \binom{4n+1}{n-1}$. In 2008, the $m$-Tamari lattice was introduced in \cite{bergeron-preville} and it was conjectured that the number of intervals and labelled intervals in the $m$-Tamari lattice are given by $\frac{m+1}{n(mn+1)} \binom{(m+1)^2 n+m}{n-1}$ and $(m+1)^n (mn+1)^{n-2}$, respectively. A labelled interval in the $m$-Tamari lattice is simply an interval where the top path is decorated as a $m$-parking function.  Refinements of these two results were proven in \cite{bousquet-fusy-preville,MbmGcLfpr}. 

The duality that is proved in this article shows that the number of intervals in Tam$((N^m E)^n)$ is the same as in the $m$-Tamari lattice Tam$({(NE^m)^n})$. Using refinements and calculations, it seems that the number of labelled intervals in Tam$({(NE^m)^n})$ is equal to the number of labelled intervals on east steps in Tam$({(N^m E)^n})$, where the labelled intervals on east steps are defined by assigning the labels in the set $\{ 1,2,...,n \}$ on east steps of the upper path, and such that the labels on consecutive east steps are increasing. Note that for $m=1$, this is easy to prove since you can obtain without difficulty the same functional equations for both cases from recurrences. But we have not been able to do so in the case $m>1$. It would be interesting to see if the ideas presented in \cite{ChaPonsarX} could help prove this equality.

More recently, some researchers (see \cite{Hikita,ArmstRhoaWilli,armstrong-website,BeGaLeXicompo14}) have extended the combinatorics of the $\mathcal{DR}_{2,n}^{m}$ by considering paths and parking functions above the line with endpoints (0,0) and $(b,a)$, where $a,b$ are arbitrary positive integers\footnote{Note that the paths above the line with endpoints (0,0) and $(mn,n)$ are the same as the paths above the line with endpoints (0,0) and $(mn+1,n)$, this is why we use the term extension.}. They defined the combinatorial statistics area and dinv on these objects. So this rational Catalan combinatorics can be seen as the combinatorics of some possible generalizations of the spaces $\mathcal{DR}_{2,n}^{m}$. Even though these spaces have not yet been shown to exist, some preliminary calculations \cite{BeGaLeXicompo14} are hopeful that they do. One might try now to define a dinv statistic on paths and parking functions above an arbitrary path consisting of east and north steps, even if it is unknown to be possible. 

It remains to be seen if our lattices Tam$(v)$, for arbitrary paths $v$, will give a combinatorial setup for the not yet defined generalizations\footnote{As it is explained in the previous paragraph, at the moment it is not known if these generalizations exist.}  of the spaces $\mathcal{DR}_{3,n}^{m}$. It will be interesting to verify this as the theory of the ``generalized'' diagonal coinvariant spaces develops.



In this article, we showed that the the lattices Tam$(v)$ based on paths and the lattices I$(v)$ based on trees with the same canopy are equivalent (i.e. isomorphic). We would like to mention that we know a third combinatorial model that is isomorphic to these two. We give only a short description here. In \cite{bousquet-fusy-preville}, an anonymous referee shows a combinatorial model for the $m$-Tamari lattice based on $(m+1)$-ary trees. The same idea can be used to define a combinatorial model isomorphic to the Tam$(v)$'s. A planted rooted tree is a rooted tree such that the children of any vertex are totally ordered. Let $T$ be a planted rooted tree. We define the prefix order sequence of $T$ to be the sequence of degrees of the internal vertices of $T$ in prefix order. Consider the set of planted rooted trees that have the same prefix order sequence. To define the covering relation on them, one first chooses a leaf $l$ that is followed (in prefix order) by an internal vertex $s$, of degree $k$. Then by denoting by $T_1, T_2, ..., T_k$ the $k$ subtrees attached to $s$, from left to right, we insert $s$ with its first $k-1$ subtrees in place of the leaf $l$ and $l$ becomes the rightmost child of $s$. The rightmost subtree of $s$, $T_k$, finally takes the former place of $s$. This combinatorial model (a lattice) is equivalent to the Tam$(v)$'s.
It will be interesting to see if there are other combinatorial models equivalent to the Tam$(v)$'s.

We finish this article by mentioning that in a forthcoming paper \cite{Lfprcaninterv14}, it will be shown that the total number of intervals in the lattices  Tam$(v)$,  for all the paths  $v$  of length $n$, is given by $ \frac{2 (3n+3)!} {(n+2)!(2n+3)!}$, which is the same as the number of rooted non-separable planar maps with $n+2$ edges. 

\bigskip
\noindent{\bf Acknowledgements.} 
We are very grateful to Luc Lapointe for making this collaboration possible by making the two authors meet on several occasions, for many fruitful discussions, and for his constant support during the entire process of this work. LFPR would like to thank François Bergeron for many fruitful discussions, in particular for providing the first author a conjectural formula that was one of his motivations for this work, and to Ruben Escobar and Miles Jones for multiple corrections. XV is grateful to Kurusch, Vincent Pilaud and Juanjo Ru\'e for inviting him to the workshop in Madrid (27-29 november 2013)  "Recent trends in algebraic and geometric combinatorics" to present part of this work, and to Mireille Bousquet-M\'elou, Michael Drmota, Christian Krattenthaler and Marc Noy for inviting him at the Oberwolfach workshop (3-7 march 2014) "Enumerative Combinatorics" to present the work described in this paper.

\bibliographystyle{plain}
\bibliography{LFPR_XV_extensionTamarilattice}

\spacebreak

\end{document}